\documentclass[11pt,twoside]{article}
\usepackage[T1]{fontenc}
\usepackage[hmargin=1.25in,vmargin=1.25in, a4paper, centering]{geometry}
\usepackage[nottoc]{tocbibind}

\usepackage[biblabel,nosort,nocompress]{cite}

\usepackage{graphicx}
\usepackage{amsmath,amsthm}
\usepackage{amssymb}
\usepackage{upgreek}
\usepackage{biolinum}
\usepackage{garamondlibre}
\usepackage[notext,noDcommand,partialup]{kpfonts}
\usepackage[cal=stixplain,frak=euler]{mathalpha}
\usepackage{courierten}

\usepackage{physics}
\usepackage{bbm}
\usepackage{graphicx}
\usepackage{mathtools,slashed}
\usepackage{cases}
\usepackage{xfrac}

\usepackage[dvipsnames]{xcolor}
\usepackage[colorlinks=true, pdfstartview=FitV, linkcolor=RoyalBlue,citecolor=ForestGreen, urlcolor=BurntOrange]{hyperref}

\usepackage{tikz}
\usetikzlibrary{hobby}
\usepackage{fancyhdr}
\renewcommand\footnotemark{}
\usepackage[sf,sl,outermarks]{titlesec}
\titleformat{\section}
{\Large\bfseries\sffamily}{\filcenter\Large{\thesection.}}{1em}{\filcenter}
\titleformat{\subsection}
{\large\bfseries\sffamily}{\thesubsection.}{1em}{}
\titleformat{\subsubsection}[runin]
{\normalsize\bfseries\itshape\sffamily}{{\normalfont\bfseries\itshape\S}\thesubsubsection.}{0.5em}{}[.\hspace*{0.5ex}]
\titleformat{\paragraph}[runin]
{\normalsize\bfseries\itshape\sffamily}{{\normalfont\bfseries\itshape\S}\theparagraph.}{0.5em}{}[.\hspace*{0.5ex}]

\usepackage{abstract}

\theoremstyle{definition}
\newtheorem{defi}{\sffamily Definition}[section]

\theoremstyle{plain}
\newtheorem{theorem}[defi]{\sffamily Theorem}

\newtheorem{lemma}[defi]{\sffamily Lemma}

\makeatletter
\providecommand{\proofnamestyle}{\itshape\sffamily\bfseries}
\renewenvironment{proof}[1][\proofname]{\par
	\pushQED{\qed}%
	\normalfont \topsep6\p@\@plus6\p@\relax
	\trivlist
	\item\relax
	{\proofnamestyle
		#1\@addpunct{.}}\hspace\labelsep\ignorespaces
}{%
	\par \popQED\endtrivlist\@endpefalse
}
\makeatother

\numberwithin{equation}{section}

\newcommand*{\R}{\mathbb{R}}

\providecommand\given{} 
\newcommand\SetSymbol[1][]{\nonscript\:#1\vert \allowbreak \nonscript\: \mathopen{}}
\DeclarePairedDelimiterX\Set[1]\{\}{ \renewcommand\given{\SetSymbol[\delimsize]} #1 }

\newcommand*{\pd}{\partial}

\begin{document}
\title{\LARGE\bfseries\sffamily Local Well-posedness of Free Boundary Problems for the Euler--Monge--Amp\`ere Equations
\footnote{{\bf\sffamily Date: }\today.}
}
\author{
	\scshape Sicheng LIU
	\thanks{%
		\textbf{\sffamily Sicheng LIU: } School of Mathematics and Statistics, Ningbo University, Ningbo, 315211, Zhejiang, China.
		\textit{E-mail}: {\color{BurntOrange} \texttt{liusicheng@nbu.edu.cn}}
	}
	\and
	\scshape Tao LUO
	\thanks{%
		\textbf{\sffamily Tao LUO: }  Department of Mathematics, City University of Hong Kong, Kowloon, Hong Kong SAR, China.
		\textit{E-mail}: {\color{BurntOrange} \ttfamily taoluo@cityu.edu.hk}
	}
}
\date{}
\maketitle

\begin{abstract}
We study the free-boundary Euler--Monge--Ampère system as a conservative relaxation of ideal incompressible fluid motion. The force is determined on the actual fluid domain by a Monge--Ampère equation coupled to a nonlinear boundary condition. For each fixed relaxation parameter, we establish local well-posedness for uniformly convex initial domains of class $C^{2,\alpha}$, densities in $C^{1,\alpha}$, and arbitrary initial velocities in $C^{2,\alpha}$. The solution retains these spatial regularities, and the force potential belongs to $C^{3,\alpha}$ up to the moving boundary. Within uniformly bounded data families, the solution depends continuously on the initial data at every lower Hölder exponent.

\end{abstract}

{\sffamily\tableofcontents}
\addtocontents{toc}{\protect\setcounter{tocdepth}{2}}

\section{Introduction}\label{sec:introduction}

The free-boundary problem for the Euler--Monge--Ampère equations is studied in a bounded, moving domain $\Omega_t\subset\R^d$. The fluid is described by its density $\rho$ and velocity $\vb v$, while its acceleration is determined by a potential $q$ through a nonlinear elliptic problem on the current domain. For a fixed parameter $\delta>0$, the equations are given by
\begin{subequations}\label{EMA prob}
	\begin{equation}
		\begin{cases}
			\pd_t \rho + \div(\rho\vb v)=0,\\
			\pd_t \vb v + (\vb v\vdot\grad)\vb v = \grad q,\\
			\det(I + \delta^2\grad^2 q) = \rho,
		\end{cases}
		\qin \Omega_t,
		\label{EMA}
	\end{equation}
	together with the boundary conditions
	\begin{equation}
		\begin{cases}
			V_{\pd\Omega_t} = \vb v\vdot \vb n, \\[1ex]
			q + \dfrac{\delta^2}{2}\abs{\grad q}^2 = 0,
		\end{cases}
		\qq{on} \pd\Omega_t.
		\label{BC}
	\end{equation}
\end{subequations}
Here $\vb n$ is the unit outer normal to $\pd\Omega_t$, and $V_{\pd\Omega_t}$ is its outward normal speed. The first boundary condition states that the free surface is transported by the fluid. The second condition determines the boundary value of the force potential and fixes its additive constant. The Monge--Ampère equation is considered on the elliptic branch
\begin{equation*}
I+\delta^2\grad^2q>0 \qin\overline\Omega_t.
\end{equation*}
Neither the density nor the volume of the actual fluid domain is required to remain constant.

The system is motivated by a conservative relaxation of free-boundary incompressible Euler flow. Incompressibility is replaced by a restoring force, with coefficient $\delta^{-2}$, toward configurations of unit density. Optimal transportation relates the actual fluid configuration to an incompressible comparison configuration, whose image is not prescribed. The Monge--Ampère equation describes this mass-to-volume transformation. The boundary law in \eqref{BC} is the corresponding expression, in the actual fluid coordinates, of the zero boundary value of the normal potential on the comparison configuration. Thus, $q$ is not specified by a local equation of state: its gradient depends on both the density distribution and the shape of the entire fluid domain. The geometric motivation is given in \S\ref{sec:geometric-motivation}.

\subsection{Geometric Motivation}\label{sec:geometric-motivation}

The system \eqref{EMA prob} brings together two geometric descriptions of fluid motion: Arnold's interpretation of incompressible Euler flow as geodesic motion on a space of volume-preserving configurations \cite{Arn66,AK21}, and optimal transportation as a means of relaxing the incompressibility constraint \cite{Bre15}. In the fixed-domain setting, {\sc Brenier and Loeper} \cite[Sections~2.1--2.4]{BL04} used polar factorization to express a squared-distance penalization of this constraint through the Monge--Ampère equation. For a fluid with a free surface, {\sc Shatah and Zeng} \cite[Section~2.1]{SZ08} identified the corresponding configuration space and its normal bundle. The essential observation for the present model is that relaxing incompressibility must retain this free-boundary normal structure: it determines not only the force in the interior, but also the nonlinear boundary condition for $q$.

Let $\vb X(t,z)$ be the material flow map, with $\vb X(0,z)=z$ and $\pd_t\vb X=\vb v\circ\vb X$, and write $\rho_0=\rho(0,\cdot)$. Equip the space of configurations with the mass-weighted metric
\begin{equation*}
H\coloneqq L^2(\Omega_0,\rho_0\dd{z};\R^d),\qquad
\langle\vb W_1,\vb W_2\rangle_H
\coloneqq \int_{\Omega_0}\rho_0\vb W_1\vdot\vb W_2\dd{z}.
\end{equation*}
For the geometric discussion we work with smooth configurations. The incompressible configurations form the formal submanifold
\begin{equation*}
\mathscr M\coloneqq \Set*{
\vb Z\colon \overline\Omega_0\to\R^d\given 
\vb Z\text{ is an embedding},\quad \det\grad\vb Z=\rho_0
}.
\end{equation*}
Thus $\vb Z_{\#}(\rho_0\dd{z})=\mathbbm{1}_{U}\dd{y}$, where $U\coloneqq \vb Z(\Omega_0)$ and $\abs{U}=\int_{\Omega_0}\rho_0\dd{z}$. The weight $\rho_0$ simply allows the same material labels to describe both the compressible and incompressible configurations; when $\rho_0=1$, this is the usual space of volume-preserving embeddings. Crucially, the image $U$ is not prescribed.

Differentiating the Jacobian constraint shows that
\begin{equation*}
T_{\vb Z}\mathscr M
=\qty{\vb w\circ\vb Z\colon\, \div\vb w=0 \text{ in } U}.
\end{equation*}
Unlike the fixed-wall problem, there is no restriction on the normal component of $\vb w$ at the boundary. The normal space is therefore
\begin{equation}
\label{eq:normal-space}
N_{\vb Z}\mathscr M
=\qty{(\grad p)\circ\vb Z\colon p=0 \text{ on } \pd U}.
\end{equation}
Indeed, for every divergence-free $\vb w$,
\begin{equation*}
\braces{\langle}{\rangle}{(\grad p)\circ\vb Z,\vb w\circ\vb Z}_H
=\int_U\grad p\vdot\vb w\dd{y}
=\int_{\pd U}p\,\vb w\vdot\vb n_U\dd{S}=0,
\end{equation*}
where $\vb n_U$ is the outward unit normal to $\pd U$ (cf. \cite[\S~2.1]{SZ08} for more details). In the absence of surface tension and external forces, geodesic motion has acceleration in this normal space.

Following the penalization principle in \cite[\S~2.2]{BL04}, we replace the exact constraint by the potential energy
\begin{equation*}
\mathcal V_\delta(\vb X)
\coloneqq \frac{1}{2\delta^2}\operatorname{dist}_H^2(\vb X,\mathscr M)
=\frac{1}{2\delta^2}\inf_{\vb Z\in\mathscr M}
\int_{\Omega_0}\rho_0\abs{\vb X-\vb Z}^2\dd{z}.
\end{equation*}
Whenever a minimizing configuration $\vb Z=\vb Z[\vb X]$ is selected smoothly, the corresponding Newton equation and the first variation with respect to $\vb Z$ give
\begin{equation*}
\pd_{tt}\vb X=\frac{\vb Z-\vb X}{\delta^2},\qquad
\vb X-\vb Z\in N_{\vb Z}\mathscr M.
\end{equation*}
The coefficient $\delta^{-2}$ represents the stiffness of the restoring force toward incompressible configurations. This is a conservative relaxation, rather than a viscous regularization: on such a smooth branch, $\frac12\norm{\pd_t\vb X}_H^2+\mathcal V_\delta(\vb X)$ is a Hamiltonian. We use this smooth projection picture to motivate the equations, not to presume a globally defined smooth nearest-point projection onto $\mathscr M$. The boundary calculation leading to \eqref{eq:legendre-boundary-identity} requires only normal stationarity and a smooth, invertible convex transport branch. The comparison map is selected instantaneously, rather than evolved by an additional Euler equation.

At a fixed time, let $\Omega_t=\vb X(\Omega_0)$ be the actual fluid domain and $U_t=\vb Z(\Omega_0)$ its incompressible comparison domain. The normal-space formula \eqref{eq:normal-space} gives a potential $p$ on $U_t$ such that
\begin{equation*}
\vb X-\vb Z=\delta^2(\grad p)\circ\vb Z,\qquad
p=0\qq{on} \pd U_t.
\end{equation*}
Consequently, the map from the comparison configuration to the actual one is
\begin{equation*}
S(y)\coloneqq \vb X\circ\vb Z^{-1}(y)=y+\delta^2\grad p(y).
\end{equation*}
Motivated by the facts for dynamics in fixed domains (cf. \cite{Bre91,BL04}), we may assume that the map $S$ can be represented as an optimal transportation map by a convex potential. Accordingly, one can write
\begin{equation*}
\Theta(y)\coloneqq \frac12\abs*{y}^2+\delta^2p(y),\qquad
\Psi\coloneqq \Theta^*,\qquad
q(x)\coloneqq \frac{\Psi(x)-\frac12\abs{x}^2}{\delta^2}.
\end{equation*}
Here $\Theta^*$ is the Legendre transform, and we work on a branch where the mutually inverse gradient maps extend to diffeomorphisms of the closures. The source-to-comparison map is then
\begin{equation*}
\mathfrak T\coloneqq \vb Z\circ\vb X^{-1}
=S^{-1}=\grad\Psi
=\operatorname{Id}+\delta^2\grad q.
\end{equation*}
Since $\vb X_{\#}(\rho_0\dd{z})=\rho\mathbbm{1}_{\Omega_t}\dd{x}$ and $\vb Z_{\#}(\rho_0\dd{z})=\mathbbm{1}_{U_t}\dd{y}$, change of variables and convexity yield
\begin{equation*}
\mathfrak T_{\#}(\rho\mathbbm{1}_{\Omega_t}\dd{x})
=\mathbbm{1}_{U_t}\dd{y},\qquad
\det(I+\delta^2\grad^2q)=\rho,\qquad
I+\delta^2\grad^2q>0.
\end{equation*}
Moreover,
\begin{equation*}
\pd_{tt}\vb X
=\frac{\mathfrak T\circ\vb X-\vb X}{\delta^2}
=(\grad q)\circ\vb X,
\end{equation*}
which gives the momentum equation in \eqref{EMA}. Transport of the material mass and boundary gives the continuity equation and $V_{\pd\Omega_t}=\vb v\vdot\vb n$.

It remains to express the zero boundary value of the normal potential in the actual fluid coordinates. At $y=\mathfrak T(x)$, Legendre duality gives
\begin{equation*}
\begin{split}
\delta^2\qty\big(q(x)+p(y))
&=\Psi(x)+\Theta(y)-\frac12\abs{x}^2-\frac12\abs*{y}^2\\
&=x\vdot y-\frac12\abs{x}^2-\frac12\abs*{y}^2 =-\frac12\abs*{x-y}^2.
\end{split}
\end{equation*}
Using $y-x=\delta^2\grad q(x)$, we obtain the exact identity
\begin{equation}
\label{eq:legendre-boundary-identity}
p\circ\mathfrak T
=-q-\frac{\delta^2}{2}\abs{\grad q}^2.
\end{equation}
Because $\mathfrak T(\pd\Omega_t)=\pd U_t$ and $p=0$ on $\pd U_t$, this yields
\begin{equation*}
q+\frac{\delta^2}{2}\abs{\grad q}^2=0
\qq{on} \pd\Omega_t.
\end{equation*}
Thus the quadratic term in \eqref{BC} is the exact change from the normal potential on the incompressible configuration to the force potential on the actual configuration. 

This interpretation also explains the formal connection with incompressible Euler. For a smoothly convergent, well-prepared family with uniformly controlled derivatives,
\begin{equation*}
\rho=\det(I+\delta^2\grad^2q)
=1+\delta^2\vb\Delta q+O(\delta^4),
\end{equation*}
so the continuity equation formally yields $\div\vb v=0$, while \eqref{BC} tends to $q=0$ at the free surface. The physical pressure in the limiting Euler equation is the negative of the limiting $q$. At finite $\delta$, the comparison-domain potential $p$ is instead related to $q$ by the exact pullback identity \eqref{eq:legendre-boundary-identity}; the two potentials are not functions on the same domain. A rigorous singular limit requires parameter-uniform estimates and a suitable family of initial data, which goes beyond the scope of this manuscript.

Geometrically, the model replaces the instantaneous pressure reaction enforcing incompressibility by a finite normal displacement between two fluid configurations. The Monge--Ampère equation records the associated mass-to-volume conversion and the nonlinear boundary law carries the free-surface pressure condition through Legendre duality. Unlike a local equation of state, this force depends on the entire density distribution and on the shape of the moving fluid domain. This geometric picture motivates the Wasserstein projection used in the static construction in \S\ref{sec:static-construction}. The significance of the free-boundary formulation is therefore not merely that the classical Euler--Monge--Ampère equations can be posed on a moving support. Actually, the interior force and the boundary law arise from the same configuration-space geometry. The local theory studied here addresses the analytic realization of that coupling, linking optimal transportation and nonlinear elliptic regularity to the dynamics of a freely moving fluid.

\subsection{Related Work}\label{sec:related-work}

The Euler--Monge--Ampère (EMA) equations are closely related to the pressureless Euler-Poisson system, with the Poisson equation replaced by a nonlinear mass-transport constraint. Their interpretation as a conservative approximation of incompressible flow was developed by {\sc Brenier and Loeper} \cite{BL04}. In that work, the Vlasov-Monge--Ampère system is derived from a squared-distance penalization of volume-preserving configurations. Global energy-preserving weak solutions and local smooth periodic solutions are constructed, and convergence to incompressible Euler is proved for well-prepared, nearly monokinetic data. The fluid EMA equations correspond to the monokinetic description, in which a single velocity is assigned to each position. The strong-solution theory and the quasi-neutral limit were studied by {\sc Loeper} \cite{Loe05} in $\mathbb{T}^2$ and $\mathbb{T}^3$. For well-prepared data, the density perturbation is of order $\delta^2$, and the initial velocity is close to a divergence-free Euler velocity. Uniform estimates are obtained on compact time intervals within the lifespan of the limiting smooth Euler solution, and strong convergence follows as $\delta\to0$. Ill-prepared data are also treated. A central ingredient in \cite{Loe05} is the control of the nonlinear Monge--Ampère inversion near the uniform state. The global behavior of fluid EMA has also been investigated under symmetry assumptions. {\sc Tadmor and Tan} \cite{TT22} established sharp critical thresholds for radially symmetric whole-space solutions, separating global regularity from finite-time breakdown.  More recently, {\sc Luan} \cite{Lua26} studied the radially symmetric system with velocity damping. These results concern the Cauchy problem without the boundary effects.

A related Monge--Ampère evolution arises in the semi-geostrophic equations, which model large-scale rotating flows. In geostrophic coordinates, a transported density is coupled to a convex potential through optimal transportation. {\sc Benamou and Brenier} \cite{BB98} showed the existence of global weak solutions. {\sc Loeper} \cite{Loe06} developed the periodic local theory for positive densities. Later, {\sc Ambrosio, Colombo, De Philippis and Figalli} \cite{ACDF12} constructed global distributional solutions in physical coordinates. Local smooth theories in physical coordinates have been established by {\sc Cheng, Cullen and Feldman} \cite{CCF18} for the two-dimensional periodic system with variable Coriolis parameter, and by {\sc Silini} \cite{Sil23} on smooth bounded planar domains endowed with a conformally flat metric. Free boundary problems have also been studied: {\sc Cullen, Kuna, Pelloni and Wilkinson} \cite{CKPW19} constructed global weak solutions in geostrophic coordinates for a three-dimensional domain with a free upper boundary. Other recent developments concern approximation and weak existence for semi-geostrophic flow; see \cite{CM25, BELP26} and the references therein.

For free-boundary incompressible Euler, one can consult \cite{Wu97, Wu99, CL00, Lin05, CS07, SZ08, IPTT25} and the references therein for energy estimates and local well-posedness under the Taylor sign condition.

This manuscript concerns the free boundary problems for EMA. Unlike the periodic problem or transport to a prescribed comparison region, the moving boundary must here be determined together with the nonlinear elliptic force. The condition $q+\delta^2\abs{\grad q}^2/2=0$ retains the free-surface normal structure, while its linearization couples boundary motion to the Monge--Ampère response. The analysis therefore requires a static solver with the regularity needed by the evolution, together with the stability as both the density and the domain vary. Theorems~\ref{thm:static}--\ref{thm:stability} provide this local theory for fixed $\delta>0$.

\subsection{Main Results}\label{sec:main-results}

Fix $0<\alpha<1$ and $\delta>0$. The independent initial data are the domain $\Omega_0$, the density $\rho_0$, and the velocity $\vb v_0$. The potential is determined by the elliptic equation and its boundary condition; no independent initial value for $q$ is prescribed.
\begin{theorem}[Static Monge--Ampère]\label{thm:static}
Suppose that $\Omega\subset\R^d$ is a uniformly convex bounded domain with $\pd\Omega\in C^{2,\alpha}$, and that $f\in C^{1,\alpha}(\overline\Omega)$ satisfies (here and henceforward, $\sigma_0$ is a constant)
\begin{equation*}
f-1\ge\sigma_0>0 \qin\overline\Omega.
\end{equation*}
Then the problem
\begin{equation}
\label{eq:static-theorem-problem}
\begin{cases}
\det(I+\grad^2u)=f \qin\Omega,\\[1ex]
u+\dfrac12\abs{\grad u}^2=0 \qq{on}\pd\Omega
\end{cases}
\end{equation}
admits a unique classical solution $u$. Moreover, there hold
\begin{equation}
\label{eq:static-theorem-bounds}
\norm{u}_{C^{3,\alpha}(\overline\Omega)}\le C_*\qc
C_*^{-1}I\le I+\grad^2u\le C_*I, \qand
\grad_{\vb n}u\ge C_*^{-1}\qq{on} \pd\Omega,
\end{equation}
where $C_*$ depends only on $d$, $\alpha$, $\sigma_0$, $\norm{f}_{C^{1,\alpha}(\overline\Omega)}$, the $C^{2,\alpha}$ geometry and uniform convexity of $\Omega$.
\end{theorem}

At each fixed time, the relation between \eqref{eq:static-theorem-problem} and \eqref{EMA prob} is given by $u=\delta^2q$ and $f=\rho$. In particular, no smallness assumption on $\delta$ is needed for this rescaling, although estimates for $q$ and the dynamical lifespan may depend on $\delta$.

\begin{theorem}[Local existence]\label{thm:existence}
Suppose that $\Omega_0\subset\R^d$ is a uniformly convex bounded domain with $\pd\Omega_0\in C^{2,\alpha}$, and that
\begin{equation*}
\rho_0\in C^{1,\alpha}(\overline\Omega_0),\qquad
\vb v_0\in C^{2,\alpha}(\overline\Omega_0;\R^d),\qquad
\rho_0-1\ge\sigma_0>0 \qin\overline\Omega_0.
\end{equation*}
Then, there exists $T_*>0$ such that \eqref{EMA prob} admits a classical solution $(\rho,\vb v,q,\Omega_t)$ on $[0,T_*]$ with the prescribed initial data. For every $t\in[0,T_*]$, the domain $\Omega_t$ is uniformly convex and
\begin{equation*}
\pd\Omega_t\in C^{2,\alpha},\qquad
\rho_t\in C^{1,\alpha}(\overline\Omega_t),\qquad
\vb v_t\in C^{2,\alpha}(\overline\Omega_t;\R^d),\qquad
q_t\in C^{3,\alpha}(\overline\Omega_t).
\end{equation*}
After reducing $T_*$ if necessary, the solution satisfies
\begin{equation*}
\rho_t-1\ge\frac{\sigma_0}{2},\qquad
I+\delta^2\grad^2q_t>0 \qin\overline\Omega_t,\qquad
\grad_{\vb n}q_t>0 \qq{on}\pd\Omega_t,
\end{equation*}
and
\begin{equation*}
\sup_{0\le t\le T_*}
\qty(
\norm{\rho_t}_{C^{1,\alpha}(\overline\Omega_t)}
+\norm{\vb v_t}_{C^{2,\alpha}(\overline\Omega_t)}
+\norm{q_t}_{C^{3,\alpha}(\overline\Omega_t)}
)\le C_*.
\end{equation*}
The constants $T_*$ and $C_*$ may be chosen in terms of $d$, $\alpha$, $\delta$, $\sigma_0$, the initial density and velocity norms, and the $C^{2,\alpha}$ geometry and uniform convexity of $\Omega_0$.

In the Lagrangian framework, let $\vb X$ be the material flow given by $\pd_t\vb X_t(\cdot) = \vb v (t, \vb X_t(\cdot))$, and let $\underline{\vb v}\coloneqq \vb v\circ\vb X$. Then, $\vb X_t\colon \overline\Omega_0\to\overline\Omega_t$ is a $C^{2,\alpha}$-diffeomorphism, with
\begin{equation*}
\vb X\in C^{1,1}_tC^{2,\alpha}_x \qand
\underline{\vb v}\in C^{0,1}_tC^{2,\alpha}_x.
\end{equation*}
For every $0<\beta<\alpha$, there also hold
\begin{equation*}
\vb X\in C^2_tC^{2,\beta}_x \qand
\underline{\vb v}\in C^1_tC^{2,\beta}_x.
\end{equation*}
Here the time interval is $[0,T_*]$ and the spatial domain is $\overline\Omega_0$. The density is recovered by
\begin{equation}
\label{eq:material-mass}
\rho_t\circ\vb X_t\,J_{\vb X_t}=\rho_0 J_{\vb X_0},
\qquad J_{\vb X_t}\coloneqq \det(\grad\vb X_t).
\end{equation}
\end{theorem}

No divergence-free or irrotationality condition is imposed on $\vb v_0$. The initial potential and its positive outer normal derivative are supplied by Theorem~\ref{thm:static}.

For solutions starting from different domains, comparison is made on the fixed reference domain $\Omega_*$ by means of the harmonic coordinate maps $\mathcal Y_{\Omega_0}$ introduced in \S\ref{sec:geometry}. For two solutions, set
\begin{equation*}
\vb X^i_0 = \text{Id} \qc \vb Y_t^i\coloneqq \vb X_t^i\circ\mathcal Y_{\Omega_0^i} \qc
\xi_t^i\coloneqq (\rho_t^i\circ\vb Y_t^i)J_{\vb Y_t^i},
\qq{where} i=1,2.
\end{equation*}
The continuity equation implies $\xi_t^i=\xi_0^i$.

\begin{theorem}[Uniqueness, stability and continuous dependence]\label{thm:stability}
Suppose that the initial domains lie in a sufficiently small common $C^{2,\alpha}$ neighbourhood of a uniformly convex reference domain, and that the two solutions satisfy uniform bounds in the classes of Theorem~\ref{thm:existence}. On a sufficiently short common lifespan, the difference energy
\begin{equation}
\label{eq:difference-energy-main}
\mathcal E(t)\coloneqq 
\norm{\vb Y_t^1-\vb Y_t^2}_{L^2(\Omega_*)}^2
+
\norm{\vb v_t^1\circ\vb Y_t^1-\vb v_t^2\circ\vb Y_t^2}_{L^2(\Omega_*)}^2
\end{equation}
satisfies
\begin{equation}
\label{eq:stability-main}
\mathcal E(t)
\le
\exp(Ct)
\qty(
\sqrt{\mathcal E(0)}
+Ct\norm{\xi_0^1-\xi_0^2}_{(H^1)'(\Omega_*)}
)^2.
\end{equation}
Here $(H^1)'(\Omega_*)$ denotes the dual of $H^1(\Omega_*)$, including test functions with nonzero boundary traces. The constant $C$ depends on the fixed parameter, the uniform regularity and geometric bounds. 

Consequently, solutions with the same initial data coincide within the class of Theorem~\ref{thm:existence}.

Continuous dependence holds at every exponent $0<\beta<\alpha$. More precisely, within a family having uniform initial bounds and a common positive density gap, convergence of
\begin{equation}
\label{eq:initial-convergence}
\mathcal Y_{\Omega_0}\text{ in } C^{2,\alpha}(\overline\Omega_*) \qc
\rho_0\circ\mathcal Y_{\Omega_0}\text{ in } C^{1,\alpha}(\overline\Omega_*), \qand
\vb v_0\circ\mathcal Y_{\Omega_0}\text{ in } C^{2,\alpha}(\overline\Omega_*)
\end{equation}
implies convergence, uniformly on a common time interval, of
\begin{equation}
\label{eq:solution-convergence}
\vb Y,\quad\vb v\circ\vb Y
\qq{in} C^{2,\beta}(\overline\Omega_*) \qand
\rho\circ\vb Y
\qq{in} C^{1,\beta}(\overline\Omega_*).
\end{equation}
The force $(\grad q)\circ\vb Y$ also converges in $C^{2,\beta}(\overline\Omega_*)$.
\end{theorem}

The static construction uses projection onto a density constraint in quadratic optimal transport. A strict separation estimate permits free-side partial-transport regularity to be transferred to the source domain without assuming convexity of the unknown transport image. The nonlinear boundary law then yields strict obliqueness and a further derivative of the potential without requiring a third derivative of the boundary. For the evolution, a polygonal construction supplies existence, while an independently justified differentiation of the static solver gives a low-order force-response estimate. Its conormal boundary structure cancels the displacement flux and controls variations in both the configuration and the transported mass.

\section{Static Monge--Ampère Problems}\label{sec:static}

In this section, we study the static problems for the Monge--Ampère equations in a uniformly convex, bounded, $C^{2, \alpha}$ domain. The problem can be written as
\begin{equation}
\begin{cases}
\det(I+\grad^2 u) = f \qin \Omega \\[1ex]
u + \abs{\grad u}^2/2 = 0 \qq{on} \pd\Omega.
\end{cases}
\label{MA static}
\end{equation}
We prove Theorem~\ref{thm:static}, with the same assumptions and quantitative conclusions as in the introduction. In particular,
\begin{equation}
f-1\ge\sigma_0>0\qq{on} \overline\Omega.
\label{f condition}
\end{equation}
The argument first constructs a solution for each datum, then proves uniqueness and data-uniform estimates.

For $d\ge2$, the transport argument in \S\ref{sec:static-construction} applies. The one-dimensional case has a direct formula: if $\Omega=(a,b)$ and $M=\int_a^b f$, set
\begin{equation*}
u'(x)=c+\int_a^x f(t)\dd{t}-(x-a),\qquad
c=\frac1M\int_a^b(x-a)f(x)\dd{x}-\frac M2,
\qquad u(a)=-\frac{c^2}{2}.
\end{equation*}
Then $1+u''=f$ and $\int_a^b u'f=0$, so the nonlinear boundary expression vanishes at both endpoints. Since $u''\ge\sigma_0$, this integral identity gives $u'(a)<0<u'(b)$. The asserted regularity and bounds follow directly; the dynamical arguments in \S\ref{sec:dynamics} are dimension-independent.

\subsection{Construction of the Static Solution}\label{sec:static-construction}

First, we consider the Wasserstein projection problems:
\begin{equation}
\inf_{\chi\in K_1} W_2^2(f \mathbbm{1}_\Omega\dd{x}, \chi\dd{y}),
\label{wass}
\end{equation}
where $K_1$ is the density cap set defined by
\begin{equation}
K_1 \coloneqq  \Set*{\chi\in L^1_+(\mathbb{R}^d) \given  0\le\chi\le1 \qc  \int_{\R^d}\chi = \int_{\Omega} f \qc \int_{\R^d} \abs*{y}^2\chi(y) \dd{y} < \infty}.
\label{def K1}
\end{equation}
Throughout, the Wasserstein distance $W_2^2$ for two measures of the same finite mass denotes the infimum of $\int\abs*{x-y}^2\dd\gamma$ over their couplings. The admissible set is nonempty, since the indicator of a ball of volume $\int_\Omega f$ belongs to $K_1$. \cite[Proposition 5.2]{DPMSV16} yields a unique minimizer $\overline\chi$ and a measurable set $E$ such that
\begin{equation*}
\overline\chi=f\mathbbm{1}_{\Omega\setminus E}+\mathbbm{1}_E.
\end{equation*}
Since $f>1$ on $\Omega$ while $\overline\chi\le1$, the set $\Omega\setminus E$ is null. Consequently
\begin{equation*}
\overline\chi=\mathbbm{1}_E\qq{a.e.,}
\abs{E}=\int_\Omega f.
\end{equation*}
We first bound the support, before asserting any global regularity of the dual potentials. Write the optimal plan from $f\mathbbm{1}_\Omega$ to $\overline\chi$ as $\gamma$. By \cite[Lemma~5.1]{DPMSV16}, for $(x,y)\in\operatorname{spt}\gamma$,
\begin{equation*}
\overline\chi=1\qq{a.e. on} B(x,\abs*{x-y}).
\end{equation*}
Thus $\omega_d\abs*{x-y}^d\le\int_\Omega f$, and
\begin{equation}
\label{eq:projection-support}
\operatorname{spt}(\overline\chi\dd{y})
\subset\overline\Omega+\overline B_{R_f}(0),\qquad
R_f\coloneqq \qty(\omega_d^{-1}\int_\Omega f)^{1/d}.
\end{equation}
Fix a ball $B_r$ containing this compact set in its interior and satisfying $\abs{B_r}>\int_\Omega f$.

\subsubsection{Dual normalization and first variation}\label{sec:dual-variation}

Quadratic-cost duality (cf. \cite{Vil03}) gives
\begin{equation*}
\frac12W_2^2(f\mathbbm{1}_\Omega\dd{x},\overline\chi\dd{y})
=\sup_{\psi(x)+\varphi(y)\le\frac12\abs*{x-y}^2}\int_\Omega\psi f\dd{x}+\int_{\R^d}\varphi\overline\chi\dd{y}.
\end{equation*}
The potential $\psi$ may be chosen Lipschitz on $\overline\Omega$, since both transported supports are bounded. Its extension to the comparison side is fixed by
\begin{equation}
\label{eq:dual-extension}
\varphi(y)\coloneqq \inf_{x\in\overline\Omega}
\qty{\frac12\abs*{x-y}^2-\psi(x)},\qquad y\in\R^d.
\end{equation}
It is locally Lipschitz and tends to $+\infty$ as $\abs*{y}\to\infty$. An arbitrary extension of a dual potential would not suffice for the subsequent variation.

For the first variation, let $\chi\in K_1$ have bounded support and put $\chi_s=(1-s)\overline\chi+s\chi$. Choose optimal potentials $(\psi_s,\varphi_s)$ with the same extension convention and normalize $\psi_s$ at one fixed point of $\Omega$. It is routine to check that $\psi_s\to\psi$ uniformly on $\overline\Omega$ and $\varphi_s\to\varphi$ locally uniformly as $s\to 0$. Using the optimal potentials at $s=0$ and $s>0$ respectively gives the two duality bounds
\begin{equation*}
\begin{split}
\int_{\R^d}\varphi(\chi-\overline\chi)\dd{y}
&\le
\frac{W_2^2(f\mathbbm{1}_\Omega,\chi_s)-W_2^2(f\mathbbm{1}_\Omega,\overline\chi)}{2s}
\\
&\le\int_{\R^d}\varphi_s(\chi-\overline\chi)\dd{y}.
\end{split}
\end{equation*}
The integrals converge because the competitor has bounded support. Minimality therefore implies
\begin{equation}
\int_{\R^d}(\overline\chi-\chi)\varphi\dd{y}\le0
\qq{for every boundedly supported} \chi\in K_1.
\label{1st var}
\end{equation}
This class of variations is sufficient for Lemma~\ref{lem:saturation}.

\begin{lemma}[The saturation threshold]\label{lem:saturation}

There is $\ell\in\R$ such that, up to null sets,
\begin{equation*}
\overline\chi=1\qq{on} \{\varphi<\ell\},\qquad
\overline\chi=0\qq{on} \{\varphi>\ell\}.
\end{equation*}

\end{lemma}

\begin{proof}
If $a<b$ and both $\{\overline\chi>0,\varphi>b\}$ and $\{\overline\chi<1,\varphi<a\}$ have positive measure, choose bounded subsets $E_+$ and $E_-$ of finite positive measure on which, respectively, $\overline\chi>\epsilon$ and $\overline\chi<1-\epsilon$, for some $\epsilon>0$. The perturbation
\begin{equation*}
h\coloneqq \frac{\mathbbm{1}_{E_-}}{\abs{E_-}}-\frac{\mathbbm{1}_{E_+}}{\abs{E_+}}
\end{equation*}
has zero mass, and $\overline\chi+sh\in K_1$ for sufficiently small $s>0$. Equation \eqref{1st var} would imply $\int h\varphi\ge0$, whereas
\begin{equation*}
\int h\varphi=\fint_{E_-}\varphi-\fint_{E_+}\varphi<a-b<0.
\end{equation*}
Consequently
\begin{equation*}
\operatorname{esssup}_{\{\overline\chi>0\}}\varphi
\le\operatorname{essinf}_{\{\overline\chi<1\}}\varphi.
\end{equation*}
The left side is finite because the occupied support is bounded, and $\varphi$ is globally bounded below. The unoccupied set has positive measure in some bounded ball, so the right side is also finite. Any intervening $\ell$ has the asserted properties. In particular,
\begin{equation}
\label{eq:saturation-threshold-mass}
\abs{\{\varphi<\ell\}}\le\int_\Omega f\le\abs{\{\varphi\le\ell\}}.
\end{equation}

\end{proof}

Subtract $\ell$ from $\varphi$ and add $\ell$ to $\psi$. This preserves duality and the extension convention, and yields
\begin{equation}
\label{eq:normalized-threshold}
\{\varphi<0\}\subset E\subset\{\varphi\le0\}
\qq{up to null sets.}
\end{equation}
Define the convex potentials
\begin{equation*}
\Theta(y)\coloneqq \frac12\abs*{y}^2-\varphi(y),\qquad
\Psi\coloneqq \Theta^*,\qquad
u(x)\coloneqq \Psi(x)-\frac12\abs{x}^2\quad(x\in\overline\Omega).
\end{equation*}
Here $\Theta$ is the supremum of affine functions indexed by $\overline\Omega$. On $\overline\Omega$, conjugacy recovers $\Psi(x)=\abs{x}^2/2-\psi(x)$. The optimal maps are (cf. \cite{Bre91})
\begin{equation*}
S=\grad\Theta,\qquad T=\grad\Psi,\qquad
T_\#(f\mathbbm{1}_\Omega\dd{x})=\mathbbm{1}_E\dd{y},\qquad
S_\#(\mathbbm{1}_E\dd{y})=f\mathbbm{1}_\Omega\dd{x}.
\end{equation*}
At this point these map identities hold almost everywhere; boundary values will follow from the next external input.

\subsubsection{The reverse partial-transport problem}\label{sec:reverse-transport}

Consider partial transport of mass $\int_\Omega f$ from $(B_r,1)$ to $(\Omega,f)$. Every feasible plan has second marginal exactly $f\mathbbm{1}_\Omega\dd{x}$, and its first marginal is a density bounded by $1$ on $B_r$. Conversely, any such density and a coupling to $f\mathbbm{1}_\Omega\dd{x}$ give a feasible partial-transport plan. The support bound \eqref{eq:projection-support} therefore identifies this minimization problem with \eqref{wass}: its active source density is $\mathbbm{1}_E$ and its active target density is $f\mathbbm{1}_\Omega$.

Since both original domains $B_r$ and $\Omega$ are strictly convex and both original densities are bounded above and away from zero, \cite[Theorems~4.8 and 4.10]{F10} yield interior strict convexity of the potentials and mutually inverse continuous gradient maps between the closed active regions. Here and below, $E$ is taken to be the canonical open active representative supplied by that result. In the present nested configuration the active target is all of $\Omega$. Thus,
\begin{equation}
\label{eq:boundary-homeomorphism}
S\colon \overline E\longrightarrow\overline\Omega,
\qquad
T\colon \overline\Omega\longrightarrow\overline E
\end{equation}
are mutually inverse homeomorphisms, with $S(E)=\Omega$ and $T(\Omega)=E$. This conclusion is obtained before any second-derivative boundary estimate. Moreover,
\begin{equation}
\label{eq:global-transport-extension}
\Theta\in C^1(\R^d),\qquad S(\R^d)\subset\overline\Omega,
\end{equation}
and $\Psi$ is $C^1$ on $\overline\Omega$ and strictly convex in $\Omega$.

In particular, for every Borel set compactly contained in $\Omega$, the measure of its gradient image equals its $f\dd{x}$-mass. Hence
\begin{equation}
\label{eq:alexandrov-equation}
\det(\grad^2\Psi)=f\qq{in} \Omega
\end{equation}
holds in the Alexandrov sense. Then, one may apply the interior $C^{2,\alpha}$ regularity theories in \cite[Theorem~1.1]{FJM16}.

On the other hand, Fenchel's inequality $\Psi(x)+\Theta(y)\ge x\cdot y$ now becomes
\begin{equation}
\varphi(y)\le u(x) + \frac12\abs*{x-y}^2 \qc x\in\overline\Omega,\quad y\in\R^d,
\label{comp phi u}
\end{equation}
and the equality holds at the optimal pairs $y = T(x) \Leftrightarrow x = S(y)$. Namely, at each differentiable point of $u$, one has
\begin{equation}
\label{eq:phi-transport-identity}
\varphi\circ T(x) = u(x) + \frac12\abs{x-T(x)}^2 = u(x) + \frac12\abs{\grad u(x)}^2 \eqqcolon  \Phi[u](x),
\end{equation}
which is the expression that must eventually vanish on $\pd\Omega$. Since $T(\Omega)=E$ and $\varphi\le0$ on the canonical open active region, one already has
\begin{equation*}
\Phi[u] \le 0 \qin\Omega \implies u\le 0 \qin \overline{\Omega}.
\end{equation*}

\subsection{Strict separation and identification of the active region}\label{sec:separation}

The Alexandrov equation \eqref{eq:alexandrov-equation} also implies the following useful viscosity inequality. Define
\begin{equation*}
c_\sigma \coloneqq  d\qty[(1+\sigma_0)^{1/d}-1]>0.
\end{equation*}
We claim that
\begin{equation}
\vb\Delta u \ge c_\sigma \qin \Omega.
\label{u vis}
\end{equation}
Indeed, the interior regularity already obtained gives $\grad^2\Psi>0$ and $\det\grad^2\Psi=f$ pointwise in $\Omega$. The arithmetic--geometric mean inequality gives
\begin{equation*}
\vb\Delta u=\operatorname{tr}(\grad^2\Psi)-d
\ge d\qty\big(f^{1/d}-1)\ge c_\sigma.
\end{equation*}
The $C^{2,\alpha}$ geometry of $\pd\Omega$ supplies a uniform interior tangent-ball radius $r_b>0$. Take $0<\eta<(1+\sigma_0)^{1/d}-1$ and $B_{r_b}(x_1)\subset\Omega$ touching $\pd\Omega$ at $\bar x_1$. Define
\begin{equation*}
w(x) \coloneqq  \frac\eta2 \qty(\abs{x-x_1}^2 - r_b^2).
\end{equation*}
Then, $\vb\Delta(u-w)>0$. Moreover, $u_{\restriction\pd B_{r_b}(x_1)} \le 0 = w_{\restriction\pd B_{r_b}(x_1)}$. The comparison principle implies $u\le w$ in $B_{r_b}(x_1)$. In particular, for $x_s \coloneqq  sx_1 + (1-s)\bar x_1$ with $0<s<1$, one has
\begin{equation*}
u(x_s) \le w(x_s) = \frac{\eta r_b^2}{2}\qty(s^2-2s).
\end{equation*}
It follows from \eqref{comp phi u} that
\begin{equation*}
\varphi(\bar x_1) \le u(x_s) + \frac12s^2r_b^2 \le \frac{\eta r_b^2}{2}\qty(s^2-2s) + \frac{r_b^2}{2} s^2 \le \frac{r_b^2}2\qty[(1+\eta)\qty(s-\frac\eta{1+\eta})^2-\frac{\eta^2}{1+\eta}].
\end{equation*}
Namely, there holds
\begin{equation*}
\varphi_{\restriction\pd\Omega} \le \frac{-r_b^2 \eta^2}{2(1+\eta)}.
\end{equation*}
The same conclusion holds uniformly throughout $\overline\Omega$. If $z$ is at distance at most $r_b$ from the boundary, a nearest boundary point and its interior tangent ball can be chosen so that $z=x_s$ for some $0\le s\le1$. Applying \eqref{comp phi u} with $x=x_{(\eta+s)/(1+\eta)}$ gives
\begin{equation*}
\varphi(z)\le
\frac{\eta r_b^2}{2(1+\eta)}\qty(s^2-2s-\eta)
\le-\frac{\eta^2r_b^2}{2(1+\eta)}.
\end{equation*}
If $z$ is at distance at least $r_b$ from the boundary, comparison on the ball of radius $r_b/2$ centred at $z$ gives $\varphi(z)\le u(z)\le-\eta r_b^2/8$. Consequently,
\begin{equation}
\label{eq:strict-separation-potential}
\varphi<-\eta_b\qq{on} \overline\Omega
\end{equation}
for a constant $\eta_b>0$ depending only on $\sigma_0$ and the geometry of $\Omega$.

Define $U\coloneqq \{\varphi<0\}$. Since $\Theta\in C^1(\R^d)$ and $S(\R^d)\subset\overline\Omega$, every point of $\{\varphi=0\}$ satisfies
\begin{equation*}
\grad\varphi(y)=y-S(y)\ne0;
\end{equation*}
the separation estimate excludes $y\in\overline\Omega$. Thus zero is a regular value of $\varphi$. Its zero set is locally a $C^1$ hypersurface, has zero Lebesgue measure, and is exactly $\pd U$. The normalized threshold identity \eqref{eq:normalized-threshold} therefore implies $\mathbbm{1}_U=\overline\chi$ almost everywhere. In particular, $U$ is bounded and $\overline\Omega\subset U$.

Since the open sets $E$ and $U$ agree almost everywhere, $\overline E=\overline U$. The $C^1$ regularity of $\pd U$ then gives $E\subset U$. By invariance of domain, the continuous injective map $S_{\restriction U}$ has open image in $\overline\Omega$, so $S(U)\subset\Omega=S(E)$. Injectivity yields $U\subset E$, and hence $E=U$. The homeomorphisms \eqref{eq:boundary-homeomorphism} therefore satisfy
\begin{equation}
	\label{eq:boundary-mapping}
	S(\pd U)=\pd\Omega,\qquad T(\pd\Omega)=\pd U.
\end{equation}

Let $L$ be the Lipschitz constant of $\varphi$ on a ball containing $\overline U$, controlled by
\eqref{eq:projection-support} and \eqref{eq:global-transport-extension}. Since $\varphi=0$
on $\pd U$, estimate \eqref{eq:strict-separation-potential} gives
\begin{equation}
	\label{eq:free-boundary-separation}
	\operatorname{dist}(\pd U,\overline\Omega)
	\ge\eta_b/L>0.
\end{equation}

\subsection{Boundary regularity of the transport potentials}\label{sec:transport-boundary}

By \eqref{eq:free-boundary-separation}, \cite[Theorem~1.1 and the remark]{CLW24} applies at every point of $\pd U$.
Moreover, \cite[\S\S~4--6]{CLW24} gives, for every $y_0\in\pd U$, a neighborhood on which $\Theta$ belongs to $C^{2,\alpha}$ up to $\pd U$ from the active side. It also gives $\pd U\in C^{2,\alpha}$ locally. Moreover, by taking the limit from the interior,
\begin{equation*}
\det(\grad^2\Theta(y_0))=\frac1{f(S(y_0))}>0.
\end{equation*}
Hence $\grad^2\Theta(y_0)$ is positive definite. By continuity the same is true with a positive lower bound in a smaller one-sided neighborhood.

Set $x_0=S(y_0)$. The local inverse theorem, applied after a $C^{2,\alpha}$ extension across the local $C^{2,\alpha}$ boundary of $U$, gives a $C^{1,\alpha}$ local inverse for $S$. The boundary homeomorphism in \eqref{eq:boundary-mapping} identifies this inverse with $T=\grad\Psi$. Thus
\begin{equation}
\label{eq:inverse-hessian}
\grad^2\Psi(x)=\qty[\grad^2\Theta(T(x))]^{-1}
\end{equation}
in a neighborhood of $x_0$ in $\overline\Omega$, and $\Psi$ belongs to $C^{2,\alpha}$ there. This is also the inverse step used at the beginning of \cite[Remark~4.4]{CLW24}. Since every point of $\pd\Omega$ is $S(y_0)$ for some $y_0\in\pd U$, a finite cover, together with the interior regularity already obtained, gives
\begin{equation}
\label{eq:source-boundary-regularity}
\Psi\in C^{2,\alpha}(\overline\Omega),\qquad
\grad^2\Psi>0\qq{on} \overline\Omega.
\end{equation}
Finally, Fenchel's equality and the boundary homeomorphism give
\begin{equation}
\label{eq:recovered-boundary-law}
\Phi[u]=\varphi\circ T=0\qq{on} \pd\Omega.
\end{equation}
Therefore $u=\Psi-\abs{x}^2/2$ is an elliptic $C^{2,\alpha}(\overline\Omega)$ solution of \eqref{MA static}.

\subsection{Higher regularity}\label{sec:static-bootstrap}

To obtain the higher regularity of $u$, one may first observe that
\begin{equation*}
\grad\Phi=\grad(u+\frac12\abs{\grad u}^2) = \grad u + \grad^2u(\grad u, \cdot) = \grad^2\Psi(\grad u, \cdot).
\end{equation*}
Set
\begin{equation*}
\mathsf{C} \coloneqq  \operatorname{cof}(\grad^2\Psi).
\end{equation*}
Then,
\begin{equation*}
\mathsf C \cdot \grad\Phi = f(\grad^2\Psi)^{-1}(\grad^2\Psi)\grad u = f \grad u,
\end{equation*}
and
\begin{equation}
\label{eq:phi-divergence}
\operatorname{div}(\mathsf C\grad\Phi)=\grad f\vdot\grad u + f\vb\Delta u \in C^{\alpha}(\overline\Omega).
\end{equation}
Piola's identity gives $\pd_i\mathsf C^{ij}=0$ in the distributional sense. Consider the zero-Dirichlet problem with operator $\mathsf C^{ij}\pd_i\pd_j$ and right-hand side $\grad f\vdot\grad u+f\vb\Delta u$. Since $\mathsf C\in C^\alpha$ is uniformly positive definite and $\pd\Omega\in C^{2,\alpha}$, this problem has a $C^{2,\alpha}$ solution by the classical Dirichlet theory \cite[\S~6.3]{GT01}. Piola's identity implies that this solution also satisfies the divergence-form equation \eqref{eq:phi-divergence}. Its difference from $\Phi$ has zero boundary values and solves the homogeneous divergence-form equation. Testing this equation by the difference and using positivity of $\mathsf C$ gives equality. Consequently, $\Phi\in C^{2,\alpha}(\overline\Omega)$.

Set
\begin{equation*}
\qty(a^{ij})\coloneqq (\grad^2\Psi)^{-1}.
\end{equation*}
Then
\begin{equation*}
a^{ij}\pd_i\pd_j\Phi = \vb\Delta u + \grad u\vdot\grad(\log f).
\end{equation*}
Moreover, $\grad\Phi=(\grad^2\Psi)\vdot\grad u \implies a[\grad(\log f), \grad\Phi]=\grad u\vdot\grad(\log f)$. Thus, one has
\begin{equation}
\label{eq:phi-elliptic}
\qty[a^{ij}\pd_i\pd_j - a^{ij}\pd_j(\log f)\pd_i]\Phi=\vb\Delta u=\tr(\grad^2\Psi)-d.
\end{equation}
The AM-GM inequality implies that $\vb*\Delta u\ge c_\sigma > 0$. Thus, the fact that $\Phi_{\restriction\pd\Omega}=0$ and the Hopf lemma yield
\begin{equation}
\label{eq:hopf-normal}
(\grad\Phi)_{\restriction\pd\Omega}=\kappa\vb n
\end{equation}
for some positive function $\kappa\colon  \pd\Omega\to\R_+$. Thus, it holds that
\begin{equation}
\label{eq:static-obliqueness}
\grad_{\vb n}u=\kappa(\grad^2\Psi)^{-1}[\vb n, \vb n] > 0 \qq{on} \pd\Omega,
\end{equation}
which means that $u$ satisfies the Taylor sign condition.

In order to lift the regularity of $u$ to $C^{3,\alpha}$, the interior Schauder estimate gives $\pd_k u\in C^{2,\alpha}_{\mathrm{loc}}(\Omega)$ and justifies
\begin{equation}
\label{eq:differentiated-interior}
a^{ij}\pd_i\pd_j\pd_k u=\pd_k(\log f)\in C^\alpha(\overline\Omega).
\end{equation}
Moreover, the fact that $\pd_k u\in C^{1,\alpha}(\overline\Omega)$ and the identity defining $\Phi$ imply
\begin{equation}
\label{eq:differentiated-boundary}
\pd_k u+\grad u\vdot\grad(\pd_k u)=\pd_k\Phi\in C^{1,\alpha}(\pd\Omega).
\end{equation}
The linear problem with these interior and boundary data has a $C^{2,\alpha}(\overline\Omega)$ solution by the oblique theory \cite[\S~6.7]{GT01}. Its boundary vector is $\grad u$, whose outer normal component is positive by the preceding Hopf argument. The difference between this solution and $\pd_k u$ belongs to $C^1(\overline\Omega)\cap C^2(\Omega)$ and solves the homogeneous problem. The maximum principle and the positive boundary zeroth-order term imply that the difference vanishes. Thus $\pd_k u\in C^{2,\alpha}(\overline\Omega)$ for each $k$, and $u\in C^{3,\alpha}(\overline\Omega)$.

Uniqueness in the admissible $C^2(\overline\Omega)$ class follows from linearization and the maximum principle.

Uniformity in \eqref{eq:static-theorem-bounds} follows by compactness of bounded data families with a fixed density gap and nondegenerate geometry.

Finally, $T=\operatorname{Id}+\grad u$ is a $C^{2,\alpha}$ diffeomorphism of the closures transporting $f\mathbbm{1}_\Omega\dd{x}$ to $\mathbbm{1}_U\dd{y}$. For a material embedding $\vb X$ with $(f\circ\vb X)J_{\vb X}=\rho_0$, the map $T\circ\vb X$ satisfies $\det\grad(T\circ\vb X)=\rho_0$ and attains the density-cap minimum. It therefore realizes the minimizing incompressible comparison configuration in \S\ref{sec:geometric-motivation} within the $C^{2,\alpha}$ embedding class.

\section{Well-posedness of Dynamical Problems}\label{sec:dynamics}

We now prove Theorems~\ref{thm:existence} and~\ref{thm:stability}. In Lagrangian coordinates, the unknowns are the flow map $\vb X$ with $\vb X(0)=\operatorname{Id}$ and the pulled-back velocity $\underline{\vb v}=\vb v\circ\vb X$. We first specify a region of configuration space on which the static force is defined, then construct a motion that stays in that region.

Choose $0 < R \ll 1$ so that every map with
\begin{equation}
\label{eq:configuration-ball}
\norm{\vb X-\operatorname{Id}}_{C^{2,\alpha}(\overline\Omega_0)}<R
\end{equation}
is an orientation-preserving diffeomorphism onto a uniformly convex domain, with common $C^{2,\alpha}$ geometry and uniform bounds for the map and its inverse. This follows directly from $C^1$-smallness for injectivity on the convex domain $\Omega_0$ and $C^2$-smallness for strict convexity of the boundary. Reduce $R$ if necessary so that it also holds that
\begin{equation}
\label{eq:configuration-density-gap}
\frac{\rho_0}{J_{\vb X}}\ge1+\frac{\sigma_0}{2}
\qq{on} \overline\Omega_0 \qand
J_{\vb X}\coloneqq \det\grad\vb X.
\end{equation}
Denote by $\mathscr X_*$ for this open $C^{2,\alpha}$ ball of radius $R/2$.

For a material motion, the continuity equation implies
\begin{equation}
\label{eq:material-density-recovery}
\pd_t(\rho_t\circ\vb X_t \cdot J_{\vb X_t}) = 0 \implies \rho_t = (\rho_0/J_{\vb X_t})\circ \vb X_t^{-1}.
\end{equation}
Thus, for $\rho_0 \in C^{1, \alpha}(\overline\Omega_0)$ and $\vb X\in \mathscr X_*$, one has $\rho \coloneqq  (\rho_0/J_{\vb X})\circ \vb X^{-1} \in C^{1, \alpha}(\overline\Omega)$. In particular, there exists a unique solution $q\in C^{3, \alpha}(\overline\Omega)$ to the static Monge--Ampère problem with nonlinear Legendre-dual boundary law
\begin{equation*}
\begin{cases*}
\det(I+\delta^2\grad^2 q) = \rho &in $\Omega$\\[1ex]
q+\frac{\delta^2}2\abs{\grad q}^2 = 0 &on $\pd\Omega$.
\end{cases*}
\end{equation*}
Set
\begin{equation}
\label{eq:material-force}
\vb f[\vb X; \rho_0] \coloneqq  (\grad q[\vb X; \rho_0])\circ \vb X.
\end{equation}
Then $\vb f$ takes values in $C^{2,\alpha}(\overline\Omega_0;\R^d)$. On the restricted neighbourhood $\mathscr X_*$, the static estimate gives
\begin{equation}
\label{eq:uniform-force}
\sup_{\vb X\in \mathscr X_*} \norm{\vb f[\vb X; \rho_0]}_{C^{2, \alpha}(\overline\Omega_0)}\le C\qty(\Omega_0,R,\norm*{\rho_0}_{C^{1,\alpha}},\sigma_0^{-1},\alpha,\delta,d).
\end{equation}

\subsection{Technical Preliminaries}\label{sec:force-compactness}

\begin{lemma}[Compactness of the material force]\label{lem:force-compactness}
Suppose that $\mathscr X_* \ni \vb X_n \to \vb X \in \mathscr X_*$ in $C^2(\overline\Omega_0)$ as $n\to\infty$. Then,
\begin{equation*}
\vb f[\vb X_n; \rho_0] \to \vb f[\vb X; \rho_0] \qin C^{2, \beta}(\overline\Omega_0) \qc \forall\, 0<\beta<\alpha.
\end{equation*}
\end{lemma}

\begin{proof}

Set
\begin{equation*}
\underline q_n \coloneqq  (q[\vb X_n; \rho_0])\circ\vb X_n \qand \vb F_n \coloneqq  \vb f[\vb X_n; \rho_0].
\end{equation*}
Then, it follows that
\begin{equation}
\label{eq:force-pullback}
\vb F_n = (\grad\vb X_n)^{-T}\grad\underline q_n
\end{equation}
and
\begin{equation}
\label{eq:hessian-pullback}
\qty(\grad^2q[\vb X_n; \rho_0])\circ\vb X_n = (\grad\vb X_n)^{-T}\vdot\qty(\grad^2\underline q_n-\sum_{1\le j \le d}F_n^j\grad^2 X_n^j)\vdot(\grad\vb X_n)^{-1}.
\end{equation}
The boundedness of $\vb X_n$ in $C^{2, \alpha}(\overline\Omega_0)$ and the static elliptic theories yield the boundedness of $\underline q_n \in C^{2, \alpha}(\overline\Omega_0) \Subset C^{2}(\overline\Omega_0)$. Namely, after passing to a subsequence, there exists a function $\underline q_\infty\in C^{2, \alpha}(\overline\Omega_0)$ so that
\begin{equation*}
\underline q_n \to \underline q_\infty \qin C^2(\overline\Omega_0).
\end{equation*}
Moreover, by definition, there holds
\begin{equation*}
\det\qty[I+\delta^2(\grad\vb X_n)^{-T}\vdot\qty(\grad^2\underline q_n-\sum_{1\le j \le d}F_n^j\grad^2 X_n^j)\vdot(\grad\vb X_n)^{-1}] = \frac{\rho_0}{J_{\vb X_n}}.
\end{equation*}
Taking $n\to\infty$, the $C^2$-convergence $\vb X_n\to \vb X$ implies that
\begin{equation*}
\det(I+\delta^2\grad^2q_\infty) = (\rho_0/J_{\vb X})\circ\vb X^{-1}, \qq{where} q_\infty\coloneqq \underline q_\infty\circ\vb X^{-1}.
\end{equation*}
Concerning the boundary data, set $\vb F_\infty\coloneqq (\grad\vb X)^{-T}\grad\underline q_\infty$. The $C^2$-convergence gives
\begin{equation*}
\underline q_\infty +\frac{\delta^2}2\abs{\vb F_\infty}^2 = 0 \qq{on}  \pd\Omega_0 \implies q_\infty + \frac{\delta^2}2\abs{\grad q_\infty}^2 = 0 \qq{on}  \vb X(\pd\Omega_0).
\end{equation*}
To invoke static uniqueness, the limiting branch must also be identified. The data-uniform bounds \eqref{eq:static-theorem-bounds} give constants $\lambda,\theta>0$ such that
\begin{equation}
\label{eq:compactness-margins}
I+\delta^2\grad^2q[\vb X_n;\rho_0]\ge\lambda I,\qquad
\grad_{\vb n_n}q[\vb X_n;\rho_0]\ge\theta.
\end{equation}
The Hessian pullback formula passes the first inequality to the limit. For the second, if $\vb n_0$ is the outward normal of $\Omega_0$, then
\begin{equation}
\label{eq:normal-pullback}
\vb n_n\circ\vb X_n
=\frac{(\grad\vb X_n)^{-T}\vb n_0}{\abs{(\grad\vb X_n)^{-T}\vb n_0}}
\end{equation}
converges uniformly to the corresponding normal for $\vb X$. Thus $q_\infty$ has the same strict ellipticity and obliqueness margins. Uniqueness in Theorem~\ref{thm:static}, which already holds in the admissible $C^{2,\alpha}$ class, now yields
\begin{equation*}
q_\infty = q[\vb X; \rho_0] \qand \vb F_n \to \vb f[\vb X;\rho_0] \qq{in}  C^1(\overline\Omega_0).
\end{equation*}
Every subsequence has the same possible limit, so the convergence holds for the full sequence. The uniform $C^{2,\alpha}$ bound for $\vb F_n$ and interpolation yield convergence in $C^{2,\beta}$ for each $\beta<\alpha$.

\end{proof}

\subsection{Construction of Solutions}\label{sec:existence}

Set $\vb X_0^h=\operatorname{Id}$ and $\underline{\vb v}_0^h=\vb v_0$. Given a time step $h>0$, write $t_k=kh$. Once $\vb X_k^h\in\mathscr X_*$ has been constructed, put
\begin{equation*}
\vb F_k^h\coloneqq \vb f[\vb X_k^h;\rho_0].
\end{equation*}
For $t=t_k+\tau$, $0\le\tau\le h$, define
\begin{equation}
\label{eq:polygonal-step}
\underline{\vb v}^h(t)\coloneqq \underline{\vb v}_k^h+\tau\vb F_k^h \qand
\vb X^h(t) \coloneqq \vb X_k^h+\tau\underline{\vb v}_k^h+\frac{\tau^2}{2}\vb F_k^h.
\end{equation}
At the right endpoint these values define the next step. Thus $\vb X^h$ is continuously differentiable and
\begin{equation*}
\pd_t\vb X^h=\underline{\vb v}^h,\qquad
\pd_t\underline{\vb v}^h=\vb f[\vb X_-^h;\rho_0]\qq{a.e.,}
\vb X_-^h(t)\coloneqq \vb X_k^h\quad(t_k\le t<t_{k+1}).
\end{equation*}
Let
\begin{equation}
M\ge\sup_{\vb X\in\mathscr X_*}
\norm{\vb f[\vb X;\rho_0]}_{C^{2,\alpha}(\overline\Omega_0)}.
\end{equation}
This is a data-dependent constant, with the map and geometry bounds included as specified above. Choose $T_*>0$ so that
\begin{equation}
\label{eq:lifespan-choice}
T_*\norm{\vb v_0}_{C^{2,\alpha}}+\frac M2T_*^2<R/2.
\end{equation}
As long as the previous grid points lie in $\mathscr X_*$, integration of the defining equations gives
\begin{equation*}
\norm{\underline{\vb v}^h(t)-\vb v_0}_{C^{2,\alpha}}\le Mt,
\qquad
\norm{\vb X^h(t)-\operatorname{Id}}_{C^{2,\alpha}}
\le t\norm{\vb v_0}_{C^{2,\alpha}}+\frac M2t^2<R/2.
\end{equation*}
Induction therefore constructs every step up to $T_*$, including all intermediate maps; the last step is truncated at $T_*$. In particular, the force is never evaluated outside its domain of definition. Additionally, there holds
\begin{equation}
\begin{split}
\norm{\underline{\vb v}^h(t)-\underline{\vb v}^h(s)}_{C^{2,\alpha}}
\le M\abs{t-s} \qand
\norm{\vb X^h(t)-\vb X^h(s)}_{C^{2,\alpha}}
\le\qty\big(\norm{\vb v_0}_{C^{2,\alpha}}+MT_*)\abs{t-s}.
\end{split}
\label{Lip X v h}
\end{equation}

The compact embeddings $C^{2,\alpha}\Subset C^{2,\beta}$ and the Arzelà--Ascoli theorem yield, along a subsequence $h\downarrow0$,
\begin{equation}
\vb X^h\longrightarrow\vb X,\qquad
\underline{\vb v}^h\longrightarrow\underline{\vb v}
\qq{in} C\qty([0,T_*];C^{2,\beta}(\overline\Omega_0))
\quad(0<\beta<\alpha).
\label{conv X v}
\end{equation}
A diagonal choice, or interpolation from one lower exponent, gives the assertion for every $\beta<\alpha$. Lower semicontinuity of Hölder seminorms retains the uniform $C^{2,\alpha}$ bounds and both Lipschitz estimates in \eqref{Lip X v h}. In particular, $\underline{\vb v}$ is Lipschitz in time with values in $C^{2,\alpha}$, and $\vb X(t)\in\mathscr X_*$ by \eqref{eq:lifespan-choice}.

The frozen and interpolated configurations satisfy the explicit estimate
\begin{equation}
\label{eq:frozen-map-error}
\sup_{0\le t\le T_*}\norm{\vb X_-^h(t)-\vb X^h(t)}_{C^{2,\alpha}}
\le h\qty\big(\norm{\vb v_0}_{C^{2,\alpha}}+MT_*).
\end{equation}
Lemma~\ref{lem:force-compactness} therefore gives
\begin{equation*}
\sup_{0\le t\le T_*}
\norm{\vb f[\vb X_-^h(t);\rho_0]-\vb f[\vb X(t);\rho_0]}_{C^{2,\beta}}
\longrightarrow0 \qas h\to 0.
\end{equation*}
For completeness, uniformity in $t$ follows by contradiction: a violating sequence of times has a convergent subsequence, and the configurations at those times converge in $C^2$ to the same limiting configuration. Lemma~\ref{lem:force-compactness} then applies to both sequences.

Passing to the limit in the integral equations gives
\begin{equation}
\vb X(t)=\operatorname{Id}+\int_0^t\underline{\vb v}(s)\dd{s},\qquad
\underline{\vb v}(t)=\vb v_0+\int_0^t\vb f[\vb X(s);\rho_0]\dd{s}
\label{int formula}
\end{equation}
in $C^{2,\beta}$ for every $\beta<\alpha$. The force is continuous in time in these spaces, again by Lemma~\ref{lem:force-compactness}, so
\begin{equation*}
\vb X\in C_t^2C_x^{2,\beta},\qquad
\underline{\vb v}\in C_t^1C_x^{2,\beta}.
\end{equation*}
To obtain the asserted endpoint time regularity, one can use the first integral identity. Since $\underline{\vb v}$ is Lipschitz with values in $C^{2,\alpha}$, its integral exists in this Banach space and is continuously differentiable there. The first equality in \eqref{int formula}, already valid in $C^{2,\beta}$, identifies that Banach-space integral with $\vb X$. Hence
\begin{equation*}
\vb X\in C_t^{1,1}C_x^{2,\alpha},\qquad
\underline{\vb v}\in C_t^{0,1}C_x^{2,\alpha}.
\end{equation*}

Define, for $x\in\Omega_t\coloneqq \vb X_t(\Omega_0)$,
\begin{equation*}
\rho(t,x)\coloneqq \qty(\frac{\rho_0}{J_{\vb X_t}})\circ\vb X_t^{-1}(x),\qquad
\vb v(t,x)\coloneqq \underline{\vb v}\qty(t,\vb X_t^{-1}(x)),\qquad
q(t,\cdot)\coloneqq q[\vb X_t;\rho_0].
\end{equation*}
The identity $\rho_t\circ\vb X_t\,J_{\vb X_t}=\rho_0$, together with
\begin{equation*}
\pd_t J_{\vb X_t}=J_{\vb X_t}\,(\operatorname{div}\vb v_t)\circ\vb X_t,
\end{equation*}
gives the continuity equation. The second identity in \eqref{int formula} gives $\pd_t\vb v+(\vb v\vdot\grad)\vb v=\grad q$. Boundary particles are transported by $\vb X_t$, so their normal speed is $\vb v\vdot\vb n$. Finally, the determinant and nonlinear boundary equations hold by the definition of the static solver. Composition estimates and Theorem~\ref{thm:static} give all the spatial bounds in Theorem~\ref{thm:existence}, with the common positive margins already enforced by the configuration ball. This proves Theorem~\ref{thm:existence}.

\subsection{Uniqueness of Solutions and Continuous Dependence on Initial Data}\label{sec:stability}

\subsubsection{Geometric Settings}\label{sec:geometry}

Fix a bounded, uniformly convex reference domain $\Omega_*$ with $\pd\Omega_*\in C^{2,\alpha}$. Choose a smooth ambient vector field $\vb*\nu_*$ transverse to $\pd\Omega_*$, for example a sufficiently close smooth approximation to its continuous outward normal. In a fixed tubular neighborhood, every sufficiently $C^1$-close boundary has a unique representation by a function $\eta_\Omega\colon \pd\Omega_*\to\R$ such that
\begin{equation}
	\label{eq:boundary-graph}
	\begin{split}
		\phi_\Omega \colon  \pd\Omega_* &\to \pd\Omega \\
		z &\mapsto z + \eta_\Omega(z)\vb*\nu_*(z)
	\end{split}
\end{equation}
is a boundary diffeomorphism. Define its harmonic extension by
\begin{equation}
	\label{eq:harmonic-coordinates}
	\begin{cases}
		\vb\Delta\mathcal{Y}_\Omega = 0 \qin \Omega_*,\\
		\mathcal Y_\Omega = \phi_\Omega \qq{on} \pd\Omega_*.
	\end{cases}
\end{equation}
The fixed transverse field and graph convention make $\mathcal Y_\Omega$ a function of the domain alone, independent of material relabelling. The Dirichlet Schauder estimate gives
\begin{equation}
	\label{eq:harmonic-estimate}
	\norm{\mathcal Y_\Omega - \operatorname{Id}}_{C^{2, \alpha}(\overline\Omega_*)} \lesssim_{\Omega_*} \abs{\phi_\Omega-\operatorname{Id}}_{C^{2, \alpha}(\pd\Omega_*)}.
\end{equation}
Taking a sufficiently small constant $0<\lambda_0\ll 1$, we denote by $\Lambda_* = \Lambda_*(\Omega_*, \lambda_0, \alpha)$ the collection of domains satisfying
\begin{equation}
	\label{eq:domain-neighbourhood}
	\abs{\phi_\Omega-\operatorname{Id}}_{C^{2, \alpha}(\pd\Omega_*)} \le \lambda_0.
\end{equation}
Then, for suitably small $\lambda_0$, the extended coordinate map $\mathcal Y_\Omega$ is a $C^{2, \alpha}$-diffeomorphism from $\Omega_*$ to $\Omega$, and the domains in $\Lambda_*$ are uniformly convex with common geometric bounds.

The dynamics is compared on the reference domain through
\begin{equation*}
\vb Y\coloneqq \vb X\circ\mathcal Y_{\Omega_0}.
\end{equation*}
Thus solutions emanating from different initial domains are represented on the same domain $\Omega_*$. The response of the static potential to a change of both the domain and the density is considered first.

Suppose that $s\mapsto\vb Y_s$ is $C^1$ with values in $C^{2,\alpha}(\overline\Omega_*;\R^d)$, and that $s\mapsto\rho_s\circ\vb Y_s$ is $C^1$ with values in $C^{1,\alpha}(\overline\Omega_*)$. Set $\Omega_s\coloneqq \vb Y_s(\Omega_*)$. The maps are assumed to be orientation-preserving diffeomorphisms, with uniform $C^{2,\alpha}$ bounds for the maps and their inverses. The domains are uniformly convex with common $C^{2,\alpha}$ geometric bounds.

Let $q_s\in C^{3,\alpha}(\overline\Omega_s)$ be the solution given by Theorem~\ref{thm:static}:
\begin{equation*}
\begin{cases}
\det(I+\delta^2\grad^2q_s)=\rho_s \qin\Omega_s,\\[1ex]
q_s+\dfrac{\delta^2}{2}\abs{\grad q_s}^2=0 \qq{on}\pd\Omega_s.
\end{cases}
\end{equation*}

\subsubsection{Dependence on the parameter}\label{sec:parameter-dependence}

First observe that
\begin{equation*}
(\grad q_s)\circ\vb Y_s
=(\grad\vb Y_s)^{-T}\grad(q_s\circ\vb Y_s)
\end{equation*}
and
\begin{equation}
\label{eq:parameter-hessian-pullback}
(\grad^2q_s)\circ\vb Y_s
=(\grad\vb Y_s)^{-T}
\qty[\grad^2(q_s\circ\vb Y_s)
-\sum_{j=1}^d\qty((\grad q_s)\circ\vb Y_s)^j\grad^2Y_s^j]\qty(\grad\vb Y_s)^{-1}.
\end{equation}
After substitution, the two equations define a continuously differentiable nonlinear operator from $C^{2,\alpha}(\overline\Omega_*)$ into
\begin{equation*}
C^\alpha(\overline\Omega_*)\times C^{1,\alpha}(\pd\Omega_*),
\end{equation*}
with parameters $\vb Y_s$ and $\rho_s\circ\vb Y_s$.  At a fixed state, the derivative with respect to the potential is the pullback of
\begin{equation*}
\qty(
\delta^2\operatorname{cof}(I+\delta^2\grad^2q_s)\mathbin{\vb{\colon}}\grad^2,
\quad 1+\delta^2\grad q_s\vdot\grad
).
\end{equation*}
This is an isomorphism from $C^{2,\alpha}(\overline\Omega_s)$ to $C^\alpha(\overline\Omega_s)\times C^{1,\alpha}(\pd\Omega_s)$ by the classical oblique theory \cite[\S~6.7]{GT01}. The implicit function theorem therefore gives
\begin{equation}
\label{eq:parameter-differentiability}
s\longmapsto q_s\circ\vb Y_s
\quad \text{of class } C^1 \text{ in } C^{2,\alpha}(\overline\Omega_*).
\end{equation}
Moreover, $s\mapsto(\grad q_s)\circ\vb Y_s$ is $C^1$ in $C^{1,\alpha}(\overline\Omega_*)$.

\subsubsection{The linearized response}\label{sec:linearized-response}

Set
\begin{equation*}
\mathfrak T_s \coloneqq  x + \delta^2\grad q_s(x).
\end{equation*}
The chain rule gives
\begin{equation*}
\grad(\mathfrak T_s\circ\vb Y_s)
=
\qty[(I+\delta^2\grad^2q_s)\circ\vb Y_s]\grad\vb Y_s.
\end{equation*}
Consequently, the Monge--Ampère equation yields
\begin{equation*}
\begin{split}
\det\grad(\mathfrak T_s\circ\vb Y_s)
&=
\qty[\det(I+\delta^2\grad^2q_s)\circ\vb Y_s]\det(\grad\vb Y_s)
\\
&=(\rho_s\circ\vb Y_s)\det(\grad\vb Y_s)
\eqqcolon \xi_s.
\end{split}
\end{equation*}
We define
\begin{equation}
\label{eq:response-h}
\dot{\vb Y}_s\coloneqq \dv{s}\vb Y_s
\qand
h_s\coloneqq \delta^2\qty[
(\grad q_s)\circ\vb Y_s\vdot\dot{\vb Y}_s
-\dv{s}(q_s\circ\vb Y_s)
].
\end{equation}
This definition is made entirely on $\Omega_*$. Differentiating
\begin{equation*}
\grad(q_s\circ\vb Y_s)
=(\grad\vb Y_s)^T[(\grad q_s)\circ\vb Y_s]
\end{equation*}
with respect to $s$ and using the definition of $h_s$, the terms containing $\grad\dot{\vb Y}_s$ cancel, giving
\begin{equation*}
\grad h_s
=
\delta^2(\grad\vb Y_s)^T
\qty[
(\grad^2q_s)\circ\vb Y_s\vdot\dot{\vb Y}_s
-\dv{s}\qty[(\grad q_s)\circ\vb Y_s]
].
\end{equation*}
Hence,
\begin{equation}
\label{eq:force-derivative-identity}
\dv{s}\qty[(\grad q_s)\circ\vb Y_s]
=
(\grad^2q_s)\circ\vb Y_s\vdot\dot{\vb Y}_s
-\delta^{-2}(\grad\vb Y_s)^{-T}\grad h_s,
\end{equation}
and therefore
\begin{equation*}
\begin{split}
\dv{s}(\mathfrak T_s\circ\vb Y_s)
&=
\dot{\vb Y}_s+\delta^2\dv{s}\qty[(\grad q_s)\circ\vb Y_s]
\\
&=
\qty[(I+\delta^2\grad^2q_s)\circ\vb Y_s]\dot{\vb Y}_s
-(\grad\vb Y_s)^{-T}\grad h_s.
\end{split}
\end{equation*}
Jacobi's formula and Piola's identity, applied to the Jacobian of the map $\mathfrak T_s\circ\vb Y_s$, yield
\begin{equation}
\label{eq:jacobi-piola-response}
\begin{split}
\dot\xi_s
&=
\operatorname{cof}\grad(\mathfrak T_s\circ\vb Y_s)
\mathbin{\vb{\colon}}\grad\dv{s}(\mathfrak T_s\circ\vb Y_s)
\\
&=
\operatorname{div}\qty{
\qty[\operatorname{cof}\grad(\mathfrak T_s\circ\vb Y_s)]^T
\dv{s}(\mathfrak T_s\circ\vb Y_s)
}.
\end{split}
\end{equation}
Moreover,
\begin{equation*}
\begin{split}
\qty[\operatorname{cof}\grad(\mathfrak T_s\circ\vb Y_s)]^T
&=
\det\grad(\mathfrak T_s\circ\vb Y_s)
\qty[\grad(\mathfrak T_s\circ\vb Y_s)]^{-1}
\\
&=
\xi_s(\grad\vb Y_s)^{-1}
\qty[(I+\delta^2\grad^2q_s)^{-1}\circ\vb Y_s].
\end{split}
\end{equation*}
For simplicity of notation, set
\begin{equation}
\label{eq:response-coefficients}
P_s\coloneqq \xi_s(\grad\vb Y_s)^{-1}
\qty\big[(I+\delta^2\grad^2q_s)^{-1}\circ\vb Y_s]
(\grad\vb Y_s)^{-T} \qand
G_s\coloneqq \xi_s(\grad\vb Y_s)^{-1}\dot{\vb Y}_s.
\end{equation}
Here $P_s$ is symmetric and uniformly positive definite. The previous calculation is exactly
\begin{equation}
\operatorname{div}(P_s\grad h_s)=\operatorname{div}G_s-\dot\xi_s
\qq{in} \Omega_*.
\label{response interior}
\end{equation}

Differentiating the pulled-back boundary identity yields
\begin{equation*}
\begin{split}
0&=\dv{s}\qty[\qty(q_s+\frac{\delta^2}{2}\abs{\grad q_s}^2)\circ\vb Y_s]\\
&=\qty\big[(I+\delta^2\grad^2q_s)\grad q_s]\circ\vb Y_s\vdot\dot{\vb Y}_s
-\delta^{-2}h_s
-\qty\big[(\grad q_s)\circ\vb Y_s]\vdot(\grad\vb Y_s)^{-T}\grad h_s.
\end{split}
\end{equation*}
Consequently,
\begin{equation}
\qty\big[(\grad q_s)\circ\vb Y_s]\vdot(\grad\vb Y_s)^{-T}\grad h_s
-\qty\big[(I+\delta^2\grad^2q_s)\grad q_s]\circ\vb Y_s\vdot\dot{\vb Y}_s
=-\delta^{-2}h_s
\qq{on} \pd\Omega_*.
\label{response oblique}
\end{equation}
This is the direct boundary derivative. Its conversion into a conormal condition uses the geometry of the original nonlinear law. Let $\vb n_s$ and $\vb N$ be the outward unit normals to $\pd\Omega_s$ and $\pd\Omega_*$, respectively. As in the static argument, there is a positive function $\kappa_s$ on $\pd\Omega_s$ such that
\begin{equation}
\label{eq:response-normal}
(I+\delta^2\grad^2q_s)\grad q_s=\kappa_s\vb n_s,\qquad
\kappa_s=
\frac{(I+\delta^2\grad^2q_s)(\grad q_s,\grad q_s)}{\grad_{\vb n_s}q_s}>0.
\end{equation}
The normal transformation is
\begin{equation*}
\vb n_s\circ\vb Y_s
=\frac{(\grad\vb Y_s)^{-T}\vb N}{\abs{(\grad\vb Y_s)^{-T}\vb N}}.
\end{equation*}
It follows, using symmetry of the inverse Hessian matrix, that
\begin{equation*}
\begin{split}
&(P_s\grad h_s-G_s)\vdot\vb N\\
&\quad=\frac{\xi_s\abs{(\grad\vb Y_s)^{-T}\vb N}}{\kappa_s\circ\vb Y_s}
\Biggl\{
\qty\big[(\grad q_s)\circ\vb Y_s]\vdot(\grad\vb Y_s)^{-T}\grad h_s-\qty\big[(I+\delta^2\grad^2q_s)\grad q_s]\circ\vb Y_s\vdot\dot{\vb Y}_s
\Biggr\}.
\end{split}
\end{equation*}
Substitution of \eqref{response oblique} yields the signed flux condition
\begin{equation}
(P_s\grad h_s-G_s)\vdot\vb N=-b_s h_s,\qquad
b_s\coloneqq \frac{\xi_s\abs{(\grad\vb Y_s)^{-T}\vb N}}{\delta^2\kappa_s\circ\vb Y_s}>0.
\label{response flux}
\end{equation}
Thus, \eqref{response interior} and \eqref{response flux} imply, for every $\zeta\in H^1(\Omega_*)$,
\begin{equation}
\int_{\Omega_*}P_s\grad h_s\vdot\grad\zeta\dd{z}
+\int_{\pd\Omega_*}b_sh_s\zeta\dd{S}
=\int_{\Omega_*}G_s\vdot\grad\zeta\dd{z}
+\langle\dot\xi_s,\zeta\rangle.
\label{response weak}
\end{equation}
Here $(H^1)'(\Omega_*)$ means the dual of the full space $H^1(\Omega_*)$. 

Next, one may observe that the coefficient $b_s$ has positive upper and lower bounds on the stated family. Indeed,
\begin{equation*}
\kappa_s=\frac{\grad_{\vb n_s}q_s}{(I+\delta^2\grad^2q_s)^{-1}(\vb n_s,\vb n_s)},
\end{equation*}
so this follows from the static margins, density bounds and map bounds. Uniform ellipticity and the trace form of Poincaré's inequality yield
\begin{equation}
\label{eq:response-coercivity}
\norm{h}_{H^1(\Omega_*)}^2
\le C\qty(\norm{\grad h}_{L^2(\Omega_*)}^2+\norm{h}_{L^2(\pd\Omega_*)}^2)
\le C\qty(\int_{\Omega_*}P_s\grad h\vdot\grad h+\int_{\pd\Omega_*}b_sh^2).
\end{equation}
Testing \eqref{response weak} by $h_s$ and using $\norm{G_s}_{L^2}\le C\norm{\dot{\vb Y}_s}_{L^2}$ implies
\begin{equation}
\label{eq:response-h-bound}
\norm{h_s}_{H^1(\Omega_*)}
\le C\qty(\norm{\dot{\vb Y}_s}_{L^2(\Omega_*)}
+\norm{\dot\xi_s}_{(H^1)'(\Omega_*)}).
\end{equation}
The force identity \eqref{eq:force-derivative-identity} yields
\begin{equation}
\norm{\dv{s}\qty\big[(\grad q_s)\circ\vb Y_s]}_{L^2(\Omega_*)}
\le C\qty(\norm{\dot{\vb Y}_s}_{L^2(\Omega_*)}
+\norm{\dot\xi_s}_{(H^1)'(\Omega_*)}).
\label{force response}
\end{equation}

\subsubsection{An admissible comparison path}\label{sec:comparison-path}

For two solutions on a common short time interval, set
\begin{equation*}
\vb Y_t^i\coloneqq \vb X_t^i\circ\mathcal{Y}_{\Omega_0^i},\qquad
\xi_t^i\coloneqq (\rho_t^i\circ\vb Y_t^i)J_{\vb Y_t^i},\qquad i=1,2.
\end{equation*}
The continuity equation implies
\begin{equation*}
\xi_t^i=\xi_0^i
=(\rho_0^i\circ\mathcal Y_{\Omega_0^i})J_{\mathcal Y_{\Omega_0^i}}.
\end{equation*}
At each fixed time, both the maps and the reference mass densities are interpolated linearly:
\begin{equation}
\label{eq:affine-comparison-path}
\begin{split}
\vb Y_t^s&\coloneqq (2-s)\vb Y_t^1+(s-1)\vb Y_t^2, \quad
\xi_t^s\coloneqq (2-s)\xi_t^1+(s-1)\xi_t^2,\\
\hat\rho_t^s&\coloneqq \frac{\xi_t^s}{J_{\vb Y_t^s}},\qquad
\rho_t^s\coloneqq \hat\rho_t^s\circ(\vb Y_t^s)^{-1},
\qquad 1\le s\le2.
\end{split}
\end{equation}
Thus the endpoint domains and densities are unchanged, while
\begin{equation}
\label{eq:comparison-derivatives}
\pd_s\vb Y_t^s=\vb Y_t^2-\vb Y_t^1,\qquad
\pd_s\xi_t^s=\xi_0^2-\xi_0^1.
\end{equation}
The intermediate states are used only for the static comparison; they are not required to solve the evolution equations.

The path is admissible after the common initial neighbourhood and lifespan are reduced. For the density gap, one first observes that the determinant is smooth on bounded matrices, and its second-order Taylor remainder gives
\begin{equation*}
\sup_{1\le s\le2}
\norm{J_{\vb Y_t^s}-(2-s)J_{\vb Y_t^1}-(s-1)J_{\vb Y_t^2}}_{C^0}
\le C\norm{\vb Y_t^2-\vb Y_t^1}_{C^1}^2.
\end{equation*}
Since $J_{\vb Y_t^s}$ is uniformly bounded away from zero and $\rho_t^i\ge1+\sigma_0/2$, further reduction of the neighbourhood yields
\begin{equation}
\label{eq:chord-density-gap}
\hat\rho_t^s
\ge
\qty(1+\frac{\sigma_0}{2})
\frac{(2-s)J_{\vb Y_t^1}+(s-1)J_{\vb Y_t^2}}{J_{\vb Y_t^s}}
\ge1+\frac{\sigma_0}{4}.
\end{equation}
Theorem~\ref{thm:static} and the parameter argument in \S\ref{sec:parameter-dependence} therefore apply at every point of the segment. Denote the resulting potentials by $q_t^s$.

Integration of the force-response estimate \eqref{force response} along \eqref{eq:affine-comparison-path} now leads to
\begin{equation}
\label{eq:endpoint-force}
\begin{split}
&\norm{(\grad q_t^2)\circ\vb Y_t^2-(\grad q_t^1)\circ\vb Y_t^1}_{L^2(\Omega_*)}
\\
&\qquad\le
\int_1^2
\norm{\pd_s\qty[(\grad q_t^s)\circ\vb Y_t^s]}_{L^2(\Omega_*)}\dd{s}
\\
&\qquad\le
C\qty(
\norm{\vb Y_t^2-\vb Y_t^1}_{L^2(\Omega_*)}
+\norm{\xi_0^2-\xi_0^1}_{(H^1)'(\Omega_*)}
).
\end{split}
\end{equation}

For completeness, the initial reference mass difference can be estimated in the canonical initial coordinates. Here $\hat\rho_0^i\coloneqq \rho_0^i\circ\mathcal Y_{\Omega_0^i}$ and $\vb Y_0^i=\mathcal Y_{\Omega_0^i}$. The determinant identity gives
\begin{equation}
\label{eq:initial-mass-difference}
\xi_0^2-\xi_0^1
=(\hat\rho_0^2-\hat\rho_0^1)J_{\vb Y_0^2}+\hat\rho_0^1\operatorname{div}
\int_1^2
\qty[\operatorname{cof}(\grad\vb Y_0^s)]^T
(\vb Y_0^2-\vb Y_0^1)\dd{s}.
\end{equation}
To make the variable-density product explicit, denote the vector field under the divergence by $G_0$. Then
\begin{equation*}
\langle\hat\rho_0^1\operatorname{div}G_0,\zeta\rangle
=-\int_{\Omega_*}G_0\vdot
\qty\big(\hat\rho_0^1\grad\zeta+\zeta\grad\hat\rho_0^1)\dd{z}
+\int_{\pd\Omega_*}\hat\rho_0^1{\zeta}G_0\vdot\vb N \dd{\sigma}.
\end{equation*}
The trace theorem and $\vb Y_0^2-\vb Y_0^1=\phi_{\Omega_0^2}-\phi_{\Omega_0^1}$ on $\pd\Omega_*$ imply
\begin{equation}
\label{eq:initial-mass-bound}
\norm{\xi_0^2-\xi_0^1}_{(H^1)'(\Omega_*)}
\lesssim
\norm{\hat\rho_0^2-\hat\rho_0^1}_{L^2(\Omega_*)}
+\norm{\mathcal Y_{\Omega_0^2}-\mathcal Y_{\Omega_0^1}}_{L^2(\Omega_*)}
+\abs{\phi_{\Omega_0^2}-\phi_{\Omega_0^1}}_{H^{-1/2}(\pd\Omega_*)}.
\end{equation}

\subsubsection{The difference energy}\label{sec:difference-energy}

Set
\begin{equation}
\label{eq:difference-energy}
\mathcal E(t)\coloneqq 
\norm{\vb Y_t^1-\vb Y_t^2}_{L^2(\Omega_*)}^2
+\norm{\vb v_t^1\circ\vb Y_t^1-\vb v_t^2\circ\vb Y_t^2}_{L^2(\Omega_*)}^2.
\end{equation}
Since $\pd_t\vb Y_t^i=\vb v_t^i\circ\vb Y_t^i$ and
\begin{equation*}
\pd_t(\vb v_t^i\circ\vb Y_t^i)=(\grad q_t^i)\circ\vb Y_t^i,
\end{equation*}
one obtains
\begin{equation*}
\begin{split}
\frac12\dv{t}\mathcal E(t)
={}&\int_{\Omega_*}(\vb Y_t^1-\vb Y_t^2)\vdot
(\vb v_t^1\circ\vb Y_t^1-\vb v_t^2\circ\vb Y_t^2)\dd{z}
\\
&+\int_{\Omega_*}
(\vb v_t^1\circ\vb Y_t^1-\vb v_t^2\circ\vb Y_t^2)\vdot
\qty[(\grad q_t^1)\circ\vb Y_t^1-(\grad q_t^2)\circ\vb Y_t^2]\dd{z}
\\
\le{}&C\mathcal E(t)
+C\norm{\xi_0^1-\xi_0^2}_{(H^1)'(\Omega_*)}\sqrt{\mathcal E(t)}.
\end{split}
\end{equation*}
Applying Gronwall's inequality to $\sqrt{\mathcal E(t)+\epsilon}$ and then letting $\epsilon\downarrow0$ yields
\begin{equation}
\label{eq:gronwall-stability}
\mathcal E(t)
\le\exp(Ct)
\qty(
\sqrt{\mathcal E(0)}
+Ct\norm{\xi_0^1-\xi_0^2}_{(H^1)'(\Omega_*)}
)^2.
\end{equation}
For identical initial data, both terms on the right vanish. The maps and pulled-back velocities therefore coincide.

Finally, suppose that the initial data converge in the topologies stated in Theorem~\ref{thm:stability}, with uniform norms, a common density gap, and a common short lifespan. Then $\mathcal E(0)$ and $\norm{\xi_0^1-\xi_0^2}_{(H^1)'}$ tend to zero, so the maps and pulled-back velocities converge uniformly in time in $L^2(\Omega_*)$. Their uniform $C^{2,\alpha}$ bounds and interpolation give convergence in $C^{2,\beta}(\overline\Omega_*)$ for every $0<\beta<\alpha$.

The density difference is recovered without differentiating the evolution:
\begin{equation}
\label{eq:density-difference}
\rho_t^1\circ\vb Y_t^1-\rho_t^2\circ\vb Y_t^2
=
\frac{\xi_0^1-\xi_0^2}{J_{\vb Y_t^1}}
+\xi_0^2\qty(\frac1{J_{\vb Y_t^1}}-\frac1{J_{\vb Y_t^2}}).
\end{equation}
The initial mass densities converge in $C^{1,\alpha}$, and the Jacobians converge in $C^{1,\beta}$ while remaining uniformly positive. Hence the pulled-back densities converge uniformly in time in $C^{1,\beta}$. The endpoint force estimate \eqref{eq:endpoint-force} gives convergence of the forces in $L^2$, and their uniform $C^{2,\alpha}$ bounds give convergence in $C^{2,\beta}$ by the same interpolation. This proves Theorem~\ref{thm:stability}.

\section*{Acknowledgments}
Liu's research is supported by the star-up research grant ZX2026000395 at Ningbo University. Luo's research is supported by a grant from the Research Grants Council of the Hong Kong Special Administrative Region, China (Project No. 11310023) .

\fancyhead[RO,LE]{\scshape\sffamily References}
\bibliographystyle{amsplain0}
{\small\bibliography{ref}}
\end{document}